\documentclass[11pt]{article}
\usepackage{amsmath, amsthm, amssymb, eucal, setspace, mathrsfs}

\makeatletter
\renewcommand\section{\@startsection{section}{1}{\z@}%
 						{-3.5ex \@plus -1ex \@minus -.2ex}% negative means
						{2ex \@plus.2ex}% 		positive means vertical skip
						{\large\bfseries}}
\renewcommand\subsection{\@ifstar
						{\setcounter{subsection}{\value{equation}}
					\@startsection{subsection}{2}{\z@}
                          {1.75ex \@plus.5ex \@minus.2ex}%
                           {-.4em}		% negative means horizontal
					\textit*}
					{\setcounter{subsection}{\value{equation}}
						\stepcounter{equation}
					\@startsection{subsection}{2}{\z@}
                          {1.75ex \@plus.5ex \@minus.2ex}%
                           {-.4em}		% negative means horizontal
					\textit}}
\def\@seccntformat#1{\@ifundefined{#1@cntformat}%
	{\csname the#1\endcsname\quad} 
	{\csname #1@cntformat\endcsname}} 
\def\section@cntformat{\thesection.~} 
\def\subsection@cntformat{(\thesubsection)\ }
\renewcommand*\l@section{\mdseries\small\@dottedtocline{1}{1.5em}{2em}}
\makeatother

\numberwithin{equation}{section}
\theoremstyle{plain}
\newtheorem{maintheorem}{Theorem}
\swapnumbers

\newtheorem{proposition}[equation]{Proposition}
\theoremstyle{definition}

\theoremstyle{remark}
\newtheorem{remark}[equation]{Remark}
\newtheorem{example}[equation]{Example}

\newcommand{\cE}{\mathscr{E}}
\newcommand{\cI}{\mathscr{I}}
\newcommand{\cK}{\mathscr{K}}
\newcommand{\cL}{\mathscr{L}}
\newcommand{\cO}{\mathscr{O}}
\newcommand{\cT}{\mathscr{T}}

\newcommand{\fra}{\mathfrak{a}}

\newcommand{\bC}{\mathbb{C}}
\newcommand{\bP}{\mathbb{P}}
\newcommand{\bZ}{\mathbb{Z}}
\newcommand{\vep}{\varepsilon}

\title{\textbf{Matrix Factorisation of Morse-Bott functions}}
\author{Constantin Teleman\footnote{UC Department of Mathematics, 970 Evans Hall, Berkeley, CA 94720}}
\date{July 26, 2018}
										
\begin{document}
\maketitle
\begin{abstract}
\noindent
For a function $W\in \bC[X]$ on a smooth algebraic variety $X$ with Morse-Bott critical locus 
$Y\subset X$, Kapustin, Rozansky and Saulina \cite{krs} suggest that the associated matrix 
factorisation category $\mathrm{MF}(X;W)$ should be equivalent to the differential 
graded category of $2$-periodic coherent complexes on $Y$ (with a topological twist from the 
normal bundle). I confirm their conjecture in the special case when the first neighbourhood 
of $Y$ in $X$ is split, and establish the corrected general statement. The answer involves 
the full Gerstenhaber structure on Hochschild cochains. This note was inspired 
by the failure of the conjecture, observed by Pomerleano and Preygel \cite{pp}, when $X$ is 
a general one-parameter deformation of a $K3$ surface $Y$. 

\noindent
\emph{Acknowledgements:} The author is most thankful to Dan Halpern-Leistner, Dan Pomerleano and 
Toly Preygel for bringing up the problem and for extended conversations on it, and to the referees 
for comments and suggestions. This work was partially supported by NSF grant DMS-1406056.
\end{abstract}

\section{Statements}

Associated to a regular function $W$ on a smooth algebraic variety $X$ is a differential 
super\footnote{We use the word \emph{grading} for a $\bZ$-grading and \emph{super} for a $\bZ/2$-grading. 
}-category $\mathrm{MF}(X;W)$ of \emph{matrix factorisations} \cite{or, or2, lp, p}. For quasi-projective 
$X$, objects in $\mathrm{MF}$ are represented by pairs of vector bundles and maps, $d_0:E_0 
\rightleftarrows E_1:d_1$, with $d_1\circ d_0 = W\cdot \mathrm{Id}_{E_0}$ and  $d_0\circ d_1 = 
W\cdot \mathrm{Id}_{E_1}$. For a choice of differential graded version of the category of coherent 
sheaves on $X$, morphisms between such pairs are represented by $2$-periodic complexes.  
Restricted away from the 
zero-fiber $W^{-1}(0)$, and (less obviously) away from the critical locus $Y$ of $W$ within 
$W^{-1}(0)$, the category is quasi-equivalent to $0$. For more general $X$, the global 
category is constructed by patching local objects via quasi-isomorphisms \cite{lp, or2}.

\subsection{Split Morse-Bott case.}
When $W$ has a single, Morse critical point $y$, $\mathrm{MF}(X;W)$ is quasi-equivalent to 
the category of Clifford super-modules based on $T_yX$, with the Hessian form $\partial^2W$ 
of $W$ (\cite{beh} and \S2 below). According to the parity of $\dim X$, Bott periodicity 
\cite{abs} reduces us to the category of super-vector spaces, or to super-modules over the rank 
one Clifford algebra $\mathrm{Cliff}(1)$. 

A tempting generalisation to Morse-Bott functions $W$ with critical locus $Y\in W^{-1}(0)$ 
would assert the equivalence of $\mathrm{MF}(X;W)$ with $\mathrm{DS}Coh(Y)$, 
the differential super-category  of $2$-periodic complexes of coherent $\cO_Y$-modules 
(or $\cO_Y\otimes \mathrm{Cliff}(1)$, if $Y$ has odd co-dimension). 
A topological correction to this guess arises from the first two normal Stiefel-Whitney classes 
of $Y$ in $X$. These classes assemble with the co-dimension parity into an element $\tau=
(\pm, w_1, w_2^c)$ of the super-Brauer group and define the $\tau$-twisted differential super-category 
$\mathrm{DS}Coh^\tau(Y)$  (see \S\ref{cliff} and Remark~\ref{grbr}). The first result is 

\begin{maintheorem} \label{splitcase}
If the first neighbourhood of $Y$ in $X$ is split, then $\mathrm{MF}(X;W)\equiv \mathrm{DS}Coh^\tau(Y)$.
\end{maintheorem}
\noindent
This was conjectured in \cite{krs}, \S B without the splitting assumption. However, while 
Theorem~\ref{splitcase} does generalise to the non-split case (Theorem~2 below), it fails as stated. 
This should not be so surprising: splitting the first neighbourhood is obstructed by a ``derived'' 
function linear along the normals to $Y$ (Remark~\ref{splitting} below), so that $Y$ is not quite 
the critical locus of $W$, in a derived sense. This is the source of the correction, and my aim 
in this note is to spell it out. 

\subsection{Presentation of $X$.} \label{xpresent}
We  describe a germ of $X$ near $Y$ as a deformation of (a germ within) its normal bundle 
$\nu:N\to Y$. In the holomorphic setting, this is controlled by 
Kodaira-Spencer-Kuranishi theory, and for convenience, I also encode it in the Dolbeault model. 
Algebraic information may be lost when $Y$ is not projective; but the relevant deformation theory 
is formal (we only need the formal neighbourhood of $Y$ in $X$ for the matrix factorisation category)  
and is thus controlled by the algebraic Hochschild complex. The reader may substitute algebraic models 
of this complex for Dolbeault forms; the statements and their proof can be adjusted, at the price of 
making the answer less explicit. 

There is no higher cohomology along $\nu$, so a full Kodaira-Spencer deformation 
datum $\varphi$ for the germ of $X$ lives in the space $\Omega^{(0,1)}(Y;\nu_*\cT(N))$ of Dolbeault 
differentials on $Y$ with coefficients in the push-down of the holomorphic tangent sheaf $\cT(N)$ 
of the total space $N$. Only the formal germ of $\varphi$ along $Y$ is relevant (to finite order, 
as will emerge from the answer). As we wish to fix the zero-section $Y$ and the normal bundle $N$, we 
choose a $\varphi$ which vanishes on $Y$, and whose vertical vector component vanishes quadratically there. 
(The normal derivative of the vertical component of $\varphi$ lives in $\Omega^{(0,1)}
(Y;N^\vee\otimes N)$, potentially deforming $N$.)  

Integrability of $\varphi$ and holomorphy of $W$ on $X$ are expressed using the Schouten bracket 
$\{\_\,,\_\}$:
\[
\bar{\partial}\varphi + \frac{1}{2} \{\varphi,\varphi\} =0\text{\ \ (Maurer-Cartan)}, \quad 
	\bar{\partial}_\varphi(W) := \bar{\partial}W + \{\varphi,W\} =0. 
\]
The special form of $\varphi$ forces $W$ to be $\bar{\partial}$-holomorphic along the fibres of $\nu$;   
but $\nu$ is not holomorphic for $\bar{\partial}_\varphi$, unless $\varphi$ is fully 
vertical. The normal Hessian form $\partial^2W$ at $Y$ is $\bar{\partial}$-holomorphic on $N$.

\subsection{Local simplification.} \label{localform}
Locally, over small Stein open subsets $U\subset Y$, we can choose holomorphic 
Morse coordinates $\{t_i\}$ on $X$, normal to $U$. Then,  $W =\frac{1}{2}\sum t^2_i$, while 
the vector component of $\varphi$ is purely horizontal with respect to the $t_i$. The Morse 
coordinates are also $\bar{\partial}$-holomorphic along each fibre of $\nu$. The algebraic analogue
holds formally near any affine $U$.

\subsection{The general Morse-Bott case.}
According to Kontsevich's formality theorem \cite{dtt}, formal deformations of 
$\mathrm{DS}Coh(Y)$ as a super-category are controlled by the Dolbeault-Hochschild 
complex $\Omega^{(0,\bullet)}\left(Y; \Lambda^\bullet \cT(Y)\right)$: to wit, the formal, even 
Maurer-Cartan solutions therein. (A minimalist account of this is given in \S2, with further 
discussion in the appendix.) 
Deformations of the twisted versions $\mathrm{DS}Coh^\tau(Y)$ are 
controlled in the same way, because $\tau$ is a locally constant twist. The Morse-Bott 
matrix factorisation question is answered by an even Maurer-Cartan element $\Phi_c\in 
\Omega^{(0,\bullet)}\left(Y; \Lambda^\bullet \cT (Y)\right)$ describing how $\mathrm{MF}(X;W)$ differs from  
$\mathrm{DS}Coh^\tau(Y)$. Let us find it. 

The shifted cotangent bundles $T^\vee Y[1]$, $T^\vee N[1]$ are (shifted) holomorphic 
symplectic manifolds, related by the Lagrangian correspondence $L(\nu):= \nu^*T^\vee Y[1]$ 
induced from the projection $N\xrightarrow{\:\nu\:} Y$. Extend all structure sheaves 
quasi-isomorphically by replacing $\cO_Y$ with its Dolbeault resolution (localised holomorphically 
along the fibres of $\nu$ and smoothly on the base $Y$). The DG algebras of functions for $T^\vee Y[1]$, 
$T^\vee N[1], L(\nu)$ become 
\begin{equation}\label{DSmanifolds}
\Omega^{(0,\bullet)}(Y; \Lambda^\bullet \cT(Y)), \quad \Omega^{(0,\bullet)}(Y; \nu_*\Lambda^\bullet\cT(N)), 
	\quad \Omega^{(0,\bullet)}(Y; \nu_*\cO(N)\otimes\Lambda^\bullet\cT(Y))
\end{equation}
with the $\bar{\partial}$-differential. If needed, we emphasize the DG structure by appending 
$\bar{\partial}$. Thus, $\Phi:= W+\varphi$ is a function on $(T^\vee N[1],\bar{\partial})$. Restrict 
$\Phi$ to the (co-isotropic) DG-submanifold $(L(\nu),\bar{\partial})$. As $W$ is Morse-Bott and $\varphi$ 
is nilpotent, the critical locus $C\subset L(\nu)$ of $\Phi$ along the projection to $T^\vee Y[1]$ 
is seen, by the Jacobian test, to be $C^\infty$ isomorphic to the base (see, for instance, Example~\ref{1par} 
for the local calculation). The critical value $\Phi_c$ becomes a function on $(T^\vee Y[1],\bar{\partial})$. 

\begin{proposition}\label{Gamma}
$\Phi_c$ is a Maurer-Cartan element of $\bigoplus_{p>1} \Omega^{(0,p)}(Y;\Lambda^p \cT(Y))$.
\end{proposition}

\begin{maintheorem}
$\mathrm{MF}(X;W)$ is equivalent to the deformation of $\mathrm{DS}Coh^\tau(Y)$ by $\Phi_c$.
\end{maintheorem}
\noindent
The proposition is a simple calculation (\S3), but Theorem~2 requires more preparation, 
and will be proved in a more conceptual form (Theorem~3 below).

\begin{remark}[Split case] \label{splitting}
The normal derivative of $\varphi$ at $Y$ is $\bar{\partial}$-closed in $ \Omega^{(0,1)}\left(Y; 
N^\vee\otimes \cT(Y)\right)$, and its class in $\mathrm{Ext}^1_Y(N; TY)$ obstructs the splitting the first 
neighbourhood of $Y$ in $X$. 
In the split case, we can remove it by a gauge transformation. Once $\varphi$ vanishes 
\emph{quadratically} at $Y$, $C$ agrees with the zero-section $T^\vee Y[1]$ and $\Phi_c=0$, so Theorem~2 
implies Theorem~\ref{splitcase}.   
\end{remark}

\begin{example} \label{1par}
Take $N = Y\times \mathrm{Spec}\,\bC[\![t]\!]$, $W=t^2/2$, and describe $X$ 
as a deformation of $N$ by the Maurer-Cartan path $\varphi(t) = \sum_{n\ge 1} \varphi_n 
t^n$, $\varphi_n\in \Omega^{(0,1)}(Y;TY)$. The critical point equation 
\[
\Phi'(t) = t + \varphi'(t) = 0
\]
may be solved degree-by-degree, as $\deg t=0$ and $\deg\varphi=2$. We get 
the solution and critical value
\[
\begin{split}
t_c &= -\varphi_1 + 2\varphi_1\varphi_2 -4\varphi_1\varphi_2^2 - 3\varphi_1^2\varphi_3 +O(8), \\
\Phi_c &= W(t_c)+\varphi(t_c) = -\frac{1}{2}\varphi_1^2 +\varphi_1^2\varphi_2 -2\varphi_1^2\varphi_2^2 
-3\varphi_1^3\varphi_3 + O(10).
\end{split}
\]
Universally, $\Phi_c$ is a power series in the $\varphi_n$, but on a fixed $Y$ it truncates to 
$\varphi$-degree $\dim Y$. One can see that $\varphi_n$ first appears in a monomial of $\varphi$-degree 
$n+1$, so $\Phi_c$ only depends on the neighbourhood of $Y$ of order $(\dim Y -1)$. Thus, $\Phi_c=0$ 
when $Y$ is a curve. For a surface, we only see $-\frac{1}{2}\varphi_1^2$, representing a Dolbeault 
class in $H^2(Y;\Lambda^2TY)$. For a 
$K3$ surface, this gives a single number obstructing the \cite{krs} conjecture. This 
obstruction was found by Pomerleano and Preygel \cite{pp}.
\end{example}

\begin{remark}[Thom isomorphism]
For an embedding $Y\subset X$ of \emph{real} manifolds, we have an isomorphism in topological $K$-theory
\[
{}^\tau K^*(Y) \xrightarrow{\ \sim\ } K^*(X,X\setminus Y),  
\]
in which the \emph{$K$-theory twisting $\tau$} lives in the very same super-Brauer group mentioned earlier; 
indeed, this partly motivated the introduction of Brauer twists in \cite{dk}. Theorems~1 and 2 can be 
seen as categorical versions of this, the novelty being the addition of a \emph{holomorphic twist} $\Phi_c$ 
on $Y$. The analogy is quite strong: when $X$ is a vector bundle over $Y$ with quadratic 
$W$, the categorical equivalence is implemented by the Atiyah-Bott-Shapiro Thom class of $K$-theory (\S2). 
Unfortunately, I do not know an explicit model for this Thom class when moving away from the quadratic 
vector bundle case, and the argument will proceed instead via abstract deformation theory.    
\end{remark}

\section{Refreshers}
I collect in this section some basic background for the reader's convenience.

\subsection{Clifford  bundles.} \label{cliff}
Take $Y$ and $N$ as before, but with $W\in\cO(N)$ quadratic and non-degenerate along the 
fibres of $N$. The bundle of super-algebras 
$\mathrm{Cliff}(N,W)$ over $Y$ is generated over $\cO_Y$ by sections of $N$, declared to be odd, 
with relations $\sigma\sigma'+\sigma'\sigma = \partial^2W/\partial\sigma\partial\sigma'$. 
This algebra is invertible over $\cO_Y$ modulo Morita equivalence, with inverse
$\mathrm{Cliff}(N,-W)$. Invertibility identifies the deformation theories of the categories 
of super-complexes of coherent $\cO_Y$-modules and $\mathrm{Cliff}(N,W)$ super-modules. 

If a global $\mathrm{Spin}^c$-module $S^\pm$ for $\mathrm{Cliff}(N)$ exists (a \emph{projective} version
always does), then its 
endomorphism algebra $\cE$ is $\cO_Y$ or $\cO_Y\otimes \mathrm{Cliff}(1)$, according to the 
parity of $\mathrm{rank}\, N$, and $S^\pm$ gives a Morita equivalence $\mathrm{Cliff}(N,W)\sim\cE$. 
Existence of $S^\pm$ is obstructed by Stiefel-Whitney 
classes of the orthogonal bundle $N$, specifically $w_1\in H^1(Y;\bZ/2)$ and the image $w_2^c$ of 
$w_2$ in $H^2(Y;\cO^\times)$. We can compare different quadratic bundles by the same Morita argument
to conclude that the category of 
super $\mathrm{Cliff}(N)$-modules depends only on the ``Brauer twist'' $\tau = (\pm, w_1, w_2^c)$ of 
the introduction. We denote it by $\mathrm{DS}Coh^\tau(Y)$. 

The twisted category can also be built as follows. Local equivalences between super-modules for 
$\mathrm{Cliff}(N)$ and $\cE$ are mediated by $\mathrm{Spin}$ modules $S^\pm$. On overlaps, the 
orthogonal group acts projectively on $S^\pm$, with parity and projective co-cycles 
classified by $w_1, w_2^c$. These define a topological action of the transition functions 
on the local $\cE$-module categories, patching them to $\mathrm{DS}Coh^\tau(Y)$. 

\begin{remark}[Graded Brauer group]\label{grbr} 
The three-component Brauer group is a variant of the one defined by Donovan and Karoubi \cite{dk} 
in  their study of \emph{twisted coefficients} for topological $K$-theory; we only replaced their 
class in (the torsion part of) $H^3(Y;\bZ)$ by a refinement in $H^2(Y;\cO^\times)$. (When the structure 
sheaf $\cO$ is $C^0$ or $C^\infty$, $H^2(Y;\cO^\times) \cong H^3(Y;\bZ)$; at the other extreme, 
\'etale cohomology is needed in the algebraic case, to capture projective obstructions of bundles.) 
The group law is slightly subtle, matching the addition formula for Stiefel-Whitney 
classes; but this will play no role for us.
\end{remark}

\subsection{Quadratic vector bundles.} 
The relation between quadratic matrix factorisations and Clifford modules, originally 
established in \cite{beh} over general ground rings, is the following 
\emph{Thom isomorphism}, a special case of Theorem~1. This generalizes the 
\emph{Kn\"orrer periodicity} theorem \cite{or}. 
\begin{proposition}\label{factorcliff}
The DS category $\mathrm{MF}(N;W)$  is quasi-equivalent to $\mathrm{DS}Coh^\tau(Y)$. 
\end{proposition}  

\begin{proof}
The equivalence is given by the Atiyah-Bott-Shapiro Thom class \cite{abs}. The graded-projective 
Spinor bundle $S^\pm$ of $N$ splits locally into its even and odd parts (swapped globally by $w_1$). 
Pulled back to the total space $N$, $S^\pm$ carries 
a pair  of endomorphisms $d_+: S^+ \rightleftarrows S^-:d_-$, which over a point $n\in N$ are the Clifford 
multiplications by the vector $n$. This $\tau$-twisted, curved complex $(S^\pm, d_\pm)$ is the kernel for 
a pair of adjoint functors between the categories $\mathrm{MF}(N;W)$ and $\mathrm{DS}Coh^\tau(Y)$,
which we claim are quasi-equivalences. 

The composition $\mathrm{DS}Coh^\tau \to \mathrm{MF} \to \mathrm{DS}Coh^\tau$ is the integral over $N$ 
of a kernel on $Y\times N\times Y$ supported on $Y\times_Y N\times_Y Y$. In the even case, this kernel 
is the bundle $\mathrm{Cliff}(N) \cong\mathrm{End}_\bC(S^\pm)$, with fibre-wise differential, over 
$n\in N$, defined by the Clifford commutation action of $n$. This resolves the zero-section $Y\subset N$ 
and so induces the identity. The odd case has an extra factor of $\mathrm{Cliff}(1)$ in 
$\mathrm{End}_\bC(S^\pm)$: this acts as the identity bi-module for the local $\cE$-bundles 
on $Y$. 

The reverse-composed kernel $\cK$ is supported on $N\times_Y N\subset N\times N$, and arises by 
tensoring over $\cE$ the Spin complexes on the two $N$-factors. It is the Spin complex for $W\oplus (-W)$ on 
$N\times_Y N$. As the other composition is the identity, $\cK$ is a projector. I also claim that 
$\cK$ specialises to the identity functor on $\mathrm{DS}Coh^\tau(N)$, under a degeneration that scales 
the curvature $W$ to zero. Finiteness of (local presentations of) matrix facorisations then implies that 
$\cK$ represents the identity throughout.

To see the degeneration, polarise $N\times_Y N$ bi-diagonally as $N_+\times_Y N_-$; $W$ places the 
$N_\pm$ in duality. This presents $\cK$ as the curved complex whose fibre over $(n_+,n_-)$ is
\[
\Lambda^{ev}N_- \:
\genfrac{}{}{0pt}{}{\raisebox{-5pt}{$\xrightarrow{\:\:\iota(n_+)+\vep(n_-)\:\:}$}}
{\raisebox{5pt}{$\xleftarrow[\:\:\iota(n_+)+\vep(n_-)\:\:]{}$}}\:
\Lambda^{odd} N_-
\] 
with $\iota, \vep$ denoting contraction and multiplication. The degeneration scales the contraction to 
zero, and leads to the Koszul resolution of the diagonal $N_+$, which represents the identity functor.
\end{proof}

\begin{remark}[Local-to-global]\label{ltg}
An alternative proof reduces the proposition to the Kn\"orrer periodicity theorem, which we recover 
locally over $Y$ by trivialising $N$. Since the functors are globally defined, global isomorphy 
is a formal consequence of the sheaf property of the categories $\mathrm{DS}Coh$ and 
$\mathrm{MF}$: quasi-equivalent global categories are obtained by patching local objects and morphisms 
via (coherent systems of) quasi-isomorphisms. In the analytic case of $\mathrm{DS}Coh$, this naturally 
leads us to the category of complexes of sheaves with coherent cohomologies, rather than complexes of 
coherent sheaves (which may be inadequate for computing the correct derived $\mathrm{Hom}$ spaces). 
A more substantive statement on the $\mathrm{MF}$ side is that the global matrix factorisations define 
a quasi-equivalent DS category \cite{lp, or2} for quasi-projective $Y$. 
\end{remark}

\begin{remark}[Clifford modules] Instead of the Spin module, we can use the Clifford bundle version of the Thom class, 
with fiber $\mathrm{Cliff}(N; W)$ and right Clifford multiplications. A Morita equivalence between 
$\mathrm{MF}(N;W)$ and $\mathrm{Cliff}(N; W)$-modules is induced by this Thom bi-module, deforming the 
traditional Koszul duality between symmetric and exterior algebras. The Brauer twist $\tau$ is 
now concealed in $\mathrm{Cliff}(N)$.
\end{remark}

\subsection{$L_\infty$ structures and maps.}\label{linf}
Let $V$ be a (cohomologically) graded vector space and denote by $V[1]$ its incarnation with the grading 
shifted down by one. An $L_\infty$ structure on $V$ is a degree-one vector field $B$ on $V[1]$ whose Lie 
action squares to zero, $\cL_B\circ \cL_B=0$. Classically, $V$ is a Lie algebra, the (purely quadratic) 
vector field $B$ is one-half of the Lie bracket $V^{\wedge 2}\to V$, and the null-square condition 
is the Jacobi identity. Slightly more generally, a linear term of $B$ is a differential on $V$. An 
$L_\infty$ map $\mu:(V,B)\to (V',B')$ is a graded (but not necessarily linear) map $V[1]\to V'[1]$ compatible 
with the vector fields. As a loop space, $V[1]$ is naturally based at $0$, so one usually requires 
that $B(0)=0$ and $\mu(0)=0$. One defines $\mu$ to be a quasi-isomorphism if $d\mu$ is so on the tangent 
space homology at $0\in V[1]$, with respect to the linearisation of $B$; this amounts to completing 
$V$ at $0$. In Theorem~3 below, we will secretly translate to base the domain spaces at the element 
$\partial^2W$. 

The \emph{Koszul-Chevalley complex} $\mathrm{Chev}^*(V,B)$ is the commutative DG algebra of functions  
on $V[1]$ (usually, taken to mean formal Taylor series at $0$) with differential $\cL_B$, and the notion 
of $L_\infty$ quasi-isomorphism reduces to that for commutative DG algebras in this localisation. In the 
case of DG Lie algebras, $\mathrm{Spec}\,\mathrm{Chev}^*(V,B)$ is a derived version of the moduli stack 
of solutions of the Maurer-Cartan equation $\partial v + \frac{1}{2}[v,v]=0$ in $V$, modulo 
formal gauge transformations. (The formula continues with higher Taylor components of $B$, if present, 
picking out the zero-locus.) A path $\gamma(t)$ of Maurer-Cartan solutions is thus 
an $L_\infty$ map from the Abelian rank-one Lie algebra $\fra^1$, in degree $1$, to $V$. 

\subsection{Formal deformations.}\label{formdef}
When $V$ is the Hochschild cochain complex of a DG linear category $C$, shifted down by one and with its 
$L_\infty$ structure, a formal but fundamental result identifies $\mathrm{Chev}^*$ with the moduli stack 
of formal deformations of $C \pmod{\mathrm{Aut}\:C}$. In particular, Maurer-Cartan paths correspond 
to $1$-parameter formal deformations. One referee indicated that the interpretation of these statements 
does not command consensus in the literature when \emph{curvings} (components of Hochschild degree zero) 
are included in the deformation class. For instance, it is shown in \cite{kl} 
that a reduction to the deformation theory of algebras is obstructed by a \emph{nilpotency} condition 
on the curving. This, however, does not affect our discussion: our deformation \emph{is} nilpotent 
in the sense of \cite{kl}, despite the presence of $W$. This is seen explicitly 
in the deformation class $\Phi_c$, but is clear \emph{a priori} from the local (on $Y$) triviality 
(Proposition~\ref{factorcliff}): deformations that arise from patching are necessarily nilpotent, over 
a base of finite homological dimension. Nonetheless, to dispel any unease that may be caused by this 
evasion, we recall  in the Appendix how one constructs one-parameter deformations from possibly curved, 
Maurer-Cartan paths. 

\subsection{Kontsevich formality (after Tamarkin).} The Hochschild-Kostant-Rosenberg theorem identifies 
$HH^*$ of a regular affine variety with the Gerstenhaber algebra of its polyvector fields. The $E_2$ Hochshild 
cohomology of the latter, which classifies its formal deformations \cite{f}, is just $\bC$, stemming from 
the connected part $B^2\bC^\times$ of the automorphism group; so DG 
enhancements of this same Gerstenhaber algebra are classified by the de Rham group $H^3(\bC^\times)$ 
(and vanish locally). Since there is no such characteristic class of complex manifolds, polyvector 
fields are the local model for the deformation Gerstenhaber complex, and its Dolbeault resolution forms 
the correct global model. (The same absence of characteristic classes makes the statement equivariant 
under the group of local coordinate changes.) This makes our computation of the deformation class 
intrinsically meaningful. 
Unfortunately, there seems to be no universal differential Gerstenhaber resolution in the algebraic 
category, whereas the Dolbeault resolution 
leads to a more explicit answer.

\section{Deformation theory proof}
The categorical Thom equivalence of Proposition~\ref{factorcliff} leads to an isomorphism between the 
respective (derived)  formal deformation stacks. This underlies the bijection between their complex 
points asserted in Theorem~2. This isomorphism stems from a certain $L_\infty$ quasi-isomorphism $\chi$ 
between the controlling differential super-Lie algebras, which we will spell out geometrically and 
algebraically. 

Geometrically  (Proposition~\ref{mc}), the critical locus $C$ is a super-submanifold of the variant of 
$T^\vee N[1]$ with modified differential $\bar{\partial} + \{\Phi,\_\}$. 
The latter is the Dolbeault resoution of the holomorphic DS manifold $T^\vee X[1]$ with differential 
$\{W,\_\}$. The inclusion of $C$ gives a quasi-isomorphism between the Hochschild cohomologies 
of the $\Phi_c$-deformed category $\mathrm{DS}Coh^\tau(Y)$ and of $\mathrm{MF}(N;W)$, which will be 
the linearisation of $\chi$ at (integrable points) $\Phi$. By degeneration to the normal bundle, we will 
improve this to show (Theorem~3) that the map induced by $\chi$ on Maurer-Cartan solutions matches the 
formal deformation theories of the Thom isomorphism. This will also prove Theorem~2.

Before spelling out the geometric refinement of Proposition~\ref{Gamma}, 
we verify the original statement, since we need the computation.

\begin{proof}[Proof of Proposition~\ref{Gamma}]
Consider first the special case of Example~\ref{1par}. 
The $(p,p)$ nature of $\Phi_c$ is clear. Next, $\Phi'(t_c)=0$, so $\Phi(t) = \Phi_c + O(t-t_c)^2$, 
and then 
\begin{equation}\label{criteq}
\bar{\partial}\Phi_c + \frac{1}{2}\{\Phi_c,\Phi_c\} = \bar{\partial}\left(\Phi(t)\right) + 
	\frac{1}{2}\{\Phi(t),\Phi(t)\} + O(t-t_c).
\end{equation}
On the right, the term written out vanishes identically in $t$, so setting $t=t_c$ shows that 
the left side is zero. For general $X$, the statement being local over $Y$, we repeat this in 
Morse coordinates (\S\ref{localform}). 
\end{proof}

\begin{proposition}\label{mc} 
The space $C$, with DS algebra of functions $\Omega^{(0,\bullet)}\left(Y;\Lambda^\bullet
\cT(Y)\right)$ and differential $\bar{\partial} + \{\Phi_c,\_\}$, is a super-submanifold of  
$(T^\vee N[1],\bar{\partial}_{\,\Phi})$. The embedding is a quasi-equivalence. 
\end{proposition}

\begin{proof} 
The statement is local over $Y$, so we use the Morse coordinates of \S\ref{localform}. 
Quasi-equivalence is clear, as the complex $(\Lambda^\bullet\cT(X), \{W,\_\})$ resolves 
locally\footnote{This fails globally: in general we deform $Y$.}  the skyscraper sheaf 
$\Lambda^\bullet\cT(Y)$ on $Y$. Next, assume for notational ease that only one coordinate 
$t$ is present. The differential $dt|_Y$ is a local frame of $N$, or a (linear) function 
$H$ on $T^\vee N[1]$, with zero-locus $L(\nu)$. The functions $H$ and $(t-t_c)$ generate 
the ideal $\cI_C$ of $C$ in $T^\vee N[1]$. We check that $\bar{\partial}_{\,\Phi} \cI_C\subset  
\cI_C$ and that $\bar{\partial}_{\,\Phi} = \bar{\partial}_{\,\Phi_c} \mod \cI_C$.

First, the Hamiltonian flow $\{H,\_\}$ acts on $L(\nu)$ by vertical translation, so 
$\bar{\partial}_{\,\Phi} H = -\{H,\Phi\}$ vanishes on the 
critical locus $C$. Next, using the notation in the proof of Prop.~\ref{Gamma}, substitute 
\[
\Phi'(t) = (t-t_c)\Phi''(t_c) + O(t-t_c)^2
\]
in the identity $\bar{\partial}\Phi'(t) + \{\Phi(t),\Phi'(t)\}\equiv 0$ to get
\[
\left(\bar{\partial}(t-t_c) + \{\Phi(t),(t-t_c)\}\right)\cdot\Phi''(t_c) = O(t-t_c).
\]
Invertibility of $\Phi''(t_c)$ implies that the ideal $(t-t_c)$ is closed under $\bar{\partial}_{\,\Phi}$, 
and so therefore is $\cI_C$. 

Finally, from $\Phi(t)-\Phi_c = O(t-t_c)^2$ we see that $\bar
{\partial}_{\,\Phi(t)} = \bar{\partial}_{\,\Phi_c} \mod \cI_C$.
 \end{proof} 

\subsection{The $L_\infty$ equivalence.} 
For the Hochschild complexes of $\mathrm{DS}Coh^\tau (Y)$ and $\mathrm{MF}(N,\partial^2W)$, we use the 
first two Dolbeault function spaces in \eqref{DSmanifolds}, but with differential $\bar\partial + 
\{\partial^2W, \_\}$ in the second. The relevant $L_\infty$ algebras are their down-shifts by one, with 
the Schouten bracket. Define the (non-linear) map 
\[
\chi:  \quad \Omega^{(0,\bullet)}(Y; \nu_*\Lambda^\bullet\cT(N)) \to \Omega^{(0,\bullet)}(Y; \Lambda^\bullet \cT(Y))
\]  
sending a function $\eta$ on $T^\vee N[1]$ to the critical value $\Phi_c$ of (the $L(\nu)$-restriction of) 
$\Phi:= \partial^2W+\eta$ along the projection $L(\nu)\to T^\vee Y[1]$ . In this definition, we treat a 
degree-zero component of $\eta$ as small or formal, so that $\partial^2W$ is the leading term in computing 
the critical locus. 

\begin{maintheorem}
$\chi$ is an $L_\infty$ quasi-isomorphism. Linearising at any $\Phi$ gives the map in 
Proposition~\ref{mc}.  Finally, $\chi$ matches the equivalence of formal deformation stacks induced 
by the Thom quasi-equivalence $\mathrm{MF}(N,\partial^2W) \sim \mathrm{DS}Coh^\tau(Y)$.
\end{maintheorem}
\noindent
The proof has three sections.

\begin{proof}[$L_\infty$ property.]
In a holomorphic local frame of $N$ with Morse coordinates $\{t_i\}$, we project one coordinate 
at a time, reducing the verification to a single $t$. At a point $\eta(t)= \varphi(t) + \psi(t)\cdot
\partial/\partial t$, 
\begin{equation}\label{firstvar}
d\chi: (\delta\varphi, \delta\psi) \mapsto \left(t_c+\varphi'(t_c)\right)\cdot\delta t_c +\delta\varphi(t_c) = 
\delta\varphi(t_c),
\end{equation}
having called $t_c$ the critical point of $\Phi(t):=\frac{1}{2}t^2 + \varphi(t)$ and $\delta t_c$ 
its first variation.

At a general \emph{even} point $\eta(t)= \varphi(t) + \psi(t)\cdot\partial/\partial t$ (valued in 
a super-commutative algebra $\bC[\varepsilon_i]$, so that the coefficients in $\varphi$ are even 
and those in $\psi$ odd), the value of the structural vector field 
is 
\[
\bar{\partial}\varphi(t) + \bar{\partial}\psi(t)\cdot\partial/\partial t - t\psi(t) + 
	\frac{1}{2}[\varphi,\varphi](t) + \psi(t)\psi'(t)\partial/\partial t - \varphi'(t)\psi(t) 
	+[\varphi,\psi](t)\partial/\partial t;
\]
applying $d\chi$ and using criticality of $t_c$ gives
\begin{equation}\label{dcrit}
(\bar{\partial}\varphi)(t_c) + \frac{1}{2}[\varphi,\varphi](t_c) - 
	\left(t_c+\varphi'(t_c)\right)\cdot\psi(t_c) = (\bar{\partial}\varphi)(t_c) + \frac{1}{2}[\varphi,\varphi](t_c).
\end{equation}
The right side agrees with $\bar{\partial}\Phi(t) + \frac{1}{2}
\{\Phi(t),\Phi(t)\}$ evaluated at $t=t_c$, and formula~\eqref{criteq} leads us to $\bar{\partial}
\Phi_c + \frac{1}{2}\{\Phi_c,\Phi_c\}$, which is the value 
of the structural vector field at $\Phi_c$.
\end{proof}

\begin{proof}[Linearisation]
This is Formula~\eqref{firstvar}, and quasi-isomorphy follows from Proposition~\ref{mc}.
\end{proof}

\begin{proof}[Matching deformed categories.]
We now identify the categories $\mathrm{MF}(X;W)$ and $\mathrm{DS}Coh^\tau(Y;\Phi_c)$ by tracking 
their degeneration to the normal cone, which scales the vertical directions of $N$ by a parameter 
$\xi$ and simultaneously scales $W$ by $\xi^{-2}$. At $\xi=0$, we start with the Thom isomorphism 
for $(N,\partial^2W)$; the singularity in $\Phi_\xi := \xi^{-2}W(\xi n) + 
\xi^*\varphi$ is removable and $\Phi_\xi\to\partial^2W$, as the vanishing assumptions on $\varphi$ 
in \S\ref{xpresent} overcome the $\xi^{-1}$ scaling of vertical cotangent vectors. 

Denote by $c_\xi$ and $\Phi_{c,\xi}:=\Phi_\xi(c_\xi)$ the critical points and values (after 
$L(\nu)$-restriction). Repeating our construcion of $\chi$ for the total deformation family 
over $Y[\![\xi]\!]$, the linearisation property shows\footnote{Explicitly: using dots 
for $d/d\xi$, we find $\dot{\Phi}_{c,\xi} = \dot{\Phi}_\xi(c_\xi) + \dot{c_\xi} \lrcorner\, 
d\Phi_\xi(c_\xi) = \dot{\Phi}_\xi(c_\xi) = d\chi\big(\dot{\Phi}_\xi\big)$, by criticality of $c_\xi$.} 
that $d\chi$ takes the deformation tangent 
vector $d\eta/d\xi$ to $d\Phi_{c,\xi}/d\xi$. Categorical equivalence follows for the $\xi$-formal 
deformations, and we must just explain why we can set $\xi=1$. 

The reason is that our $\mathrm{MF}$ category over $\bC(\!(\xi)\!)$ arises by extension of scalars 
from the original $2$-periodic one, extension which scales the $2$-periodicity isomorphism 
$\Sigma^2\simeq\mathrm{Id}$ by $\xi^2$. Indeed, ignoring the vertical 
change of variables (which preserve the sheaf of categories over $Y$), we are scaling the 
super-potential. On the side of $\mathrm{DS}Coh^\tau(Y;\Phi_c)$, the deformation 
is already formal by reason of high  Hochschild degree: implicit in its interpretation is 
the MC path which scales the degree $d$ term in $\Phi_c$ by $\xi^{d-2}$.  This is indeed what 
the degeneration $\Phi_{c,\xi}$ does, since the critical value computation is homogeneous for 
the corresponding scaling on Dolbeault forms.      
\end{proof}

\begin{remark}
The argument used some special properties of $\varphi$, and the reader may ask if Theorem~3 has been 
proved in full generality. For instance, higher poly-vector terms in $\eta$ would cause a problem 
in the geometric degeneration written. The answer is that the statement concerns formal deformations, 
which must come equipped with a Maurer-Cartan path originating at $(N,\partial^2W)$; the argument 
applies to those.
\end{remark}

\begin{remark}[Tech Note]
One may try to avoid setting $\xi=1$ by simply incorporating $\xi$ in the definition 
of the deformed categories, but that is not quite right. Realising $\mathrm{MF}$ as a formal 
deformation of a category of coherent sheaves requires $\xi\to\infty$, whereas $\Phi_c$ is formal 
near $\xi=0$, making comparison of the two problematic without the $2$-periodic interpolation.     
\end{remark}

\section*{Appendix: Review of formal deformations}
\stepcounter{section}
I now offer a construction of the formal deformation $C[\![t]\!]$ of a DG category with 
respect to a Maurer-Cartan path $\gamma(t)$ of Hochschild co-chains which may contain a curving. No claim of novelty applies, 
nor is a mathematical problem solved, as there is no intrinsic difficulty within curved deformation 
theory. The issue is rather to agree on the nature of a formal deformation. The discourse being 
logically unnecessary for this paper (cf.~\S\ref{formdef}), I do not give full details.

\subsection{Curved deformation problem.} 
Deforming an algebra deforms its category of left modules. The regular module always follows 
along, and has the distinguished property of generating the category. This example is the source of 
the comforting feeling that deformations are minor changes. Yet, absent the assumption of a surviving 
generating object, formal deformations can be quite brutal, making the theory look contentious. The 
trouble is that no particular object of a category is guaranteed to survive formal deformation 
of the latter --- and indeed, most objects do not, even in the case of 
modules over an algebra. A geometric example is a generic deformation of a projective $K3$ surface, 
which loses almost all coherent sheaves, already to first order. 

To first order, one encounters the \emph{Atiyah obstruction}. 
For a vector bundle $E\to X$ on a complex manifold, this is the contraction into $H^2\left(X;\mathrm{End}
(E)\right)$ of the Atiyah and Kodaira-Spencer classes. 
For a general $x\in C$, it is the canonical image in $\mathrm{Ext}^2(x,x)$ of the $HH^2$-deformation class. 
The latter can be defined at co-cycle level; should the obstruction vanish, a first deformation of $x$ is given 
by an element in $\mathrm{Hom}^1(x,x)$ killing the Atiyah co-cycle. For $E\to X$, this is a 
connection $(0,1)$-form on $E$ correcting the $\bar\partial$-operator, to ensure that it squares to zero. 
More interestingly, in the ($2$-periodic) case of a deformation of $X$ by a super-potential $\gamma(t)=tW$, 
a complex $[d_0:E^0\to E^1]$ survives if we can find $d^1:E^1\to E^0$ with $[d_0,d_1] = tW\cdot\mathrm{Id}$. 
(This example has no higher obstructions, and accounts for Orlov's construction of the 
matrix factorisation category, after setting $t=1$.)

For nilpotent Atiyah obstructions, one can hope to filter $x$ with a composition series of vanishing 
obstructions, and conclude that it lies in the thick closure of the unobstructed objects. 
(This is discussed in detail in \cite{kl}.) In common situations --- deformations of a ring, or a ringed 
space --- all modules have this property; 
but this is quite special. Thus, in the super-potential case, we can only recover sheaves 
supported in a finite neighborhood of $W^{-1}(0)$. This is a feature of curved life. 

\subsection{Curved formal deformations: torsion objects.}\label{torsionC}
To prevent the construction of $C[\![t]\!]$ being obscured by technicalties, we first build the 
full subcategory of $t$-torsion objects (supported on a finite neighborhood of $t=0$). 
The Maurer-Cartan path $\gamma$ defines, via its $L_\infty$ interpretation of \S\ref{linf}, 
an action on $C$ of the ``odd one-parameter group" with Lie algebra $\fra^1$. (The group is in reality 
the graded commutative and co-commutative (Hopf) algebra $U\fra^1:=\bC[\vep]$, $\deg\vep=1$.) 
This action leads to the crossed  product category $C_\vep :=U\fra^1\ltimes C$. It has the same objects 
as $C$; the morphism spaces are tensored with $U\fra^1$, but the obvious ``extension by scalars" 
composition rule (which we would find for $\gamma\equiv0$) is $A_\infty$-deformed by the path $\gamma(t)$, 
which imposes non-trivial higher commutators of $\vep$ with the morphisms of $C$. Specifically, 
the substitution $t\mapsto\partial/\partial\vep$ converts $\gamma$ to a Maurer-Cartan element 
for $C\otimes U\fra^1$, with respect to the quasi-isomorphism of Hochschild co-chain complexes 
\[
HCH^*(C\otimes U\fra^1) \simeq HCH^*(C) \otimes HCH^*(U\fra^1) \simeq HCH^*(C)\otimes U\fra^1[\partial/\partial\vep].
\]
The virtue of this converted deformation class 
is the absence of a curving, even when $\gamma$ carries one; so in describing $U\fra^1\ltimes C$, we 
are reduced to the case of curving-free formal deformations. 

Our $C_\vep$ is the full subcategory of objects in the formal deformation $C[\![t]\!]$ of $C$ which are 
supported at $t=0$; $\vep$ is the generator of $\mathrm{Ext}^1_{\bC[\![t]\!]}(\bC,\bC)$, and the reader 
will recognise our use of the Koszul duality between $\bC[\vep]$ and $\bC[\![t]\!]$.  
An object $x$ may be deformed to first order in $t$ iff the class $\vep$ survives in $C_\vep$ 
to $\mathrm{Ext}^1(x;x)$, in which case the extension classified by a surviving representative describes 
that very deformation. The $t$-torsion part of $C[\![t]\!]$ is generated from $C_\vep$ under successive 
extensions.\footnote{This can be characterised abstractly as the quotient of $C$ by $\fra^1$.} Iterated 
extensions may be obstructed by higher $A_\infty$ powers of $\vep$ into $\mathrm{Ext}^2(x,x)$. 

\begin{example}
The purely curved deformation $\gamma(t)=tW$ leads to the differential $\vep\mapsto W$ (and no 
other deformation). This ``kills'' skyscraper sheaves at points where $W\neq 0$: their endomorphism 
rings are now quasi-isomorphic to zero. A sheaf $\cE$ survives to first-order
iff $W=0$ on $\cE$.

A Kodaira-Spencer co-cycle $t\cdot\tau\in \Omega^{0,1}(X;\cT_X)$ square-bracketing to zero leads 
to the Poisson tensor $\partial/\partial\vep\wedge\tau$. This induces a commutator of $\vep$ with 
the structure sheaf, and we get a complete description of $C_\vep$ in this case as modules 
over $\Omega^{(0,\bullet}(X)[\vep]$ with this commutation relation. On a general sheaf $\cE$, 
however, the leading deformation is a differential, taking $\vep$ to the Atiyah obstruction; 
$\vep$ survives to first order precisely when the latter vanishes (its co-cycle is exact). 
\end{example}

\subsection{The full deformation.}
In concrete cases, we can build $C[\![t]\!]$ itself from (inverse) limits of torsion objects. The abstract 
construction is a variant of the one in \S\ref{torsionC}, leading to the $n$th order extensions 
$C[\![t]\!]/t^n$. All objects in the latter are $t$-torsion, so those categories are generated by 
their respective $C_\vep$; and $C[\![t]\!]$ is their limit over $n$. 

To construct the truncated versions of $C_\vep$, we merely replace in the discussion $\fra^1$ 
by the system of $L_\infty$ truncated algebras $\fra^{1}_{<n}$ whose Chevalley complexes are 
$\bC[\![t]\!]/t^n$. A minimal presentation of $\fra^1_{<n}$ has two linear generators $\vep,\eta$ in degrees 
$1$ and $2$ and a single, $n$-ary bracket $[\vep,\vep,\dots,\vep]=\eta$. The universal enveloping algebra
$U\fra^1_{<n}$, which replaces $\bC[\vep]$, is the deformation of $\bC[\vep,\eta]$ by the $n$-ary multiplication 
$m_n(\vep,\dots,\vep)=\eta$. We have \cite{th}
\[
HH^*(U\fra^{1}_{<n})\simeq H^*\left(C[\vep,\eta, t,u], d\right), \text{ with } du= t^n, d\vep= n\eta t^{n-1}, 
\] 
having written $t$ for $\partial/\partial\vep$ and $u$ for $\partial/\partial\eta$, in Hochschild degrees 
$0$ and $(-1)$, respectively. In degree $0$, we find as hoped $C[\![t]\!]/t^n$; this converts (the $n$-truncation of) 
$\gamma$ to a deformation class for the crossed product category, which, as before, is free of a curving term.

\begin{remark}
The specialisation $C_0$ of $C[\![t]\!]$ at $t=0$ need not agree with $C$ in general; it is a localised 
version of $C$, at the subcategory of objects with zero curving. This is established in \cite{kl} for 
the first-order deformation of DG categories of modules over a DG algebra of finite homological dimension. 
Formulating and proving a precise general statement is well beyond the scope of this note.
\end{remark}

\begin{remark}[Topological interpretation]
This perspective is closely related to the intepretation of curvings as $B\bZ$-actions of $C$, 
with matrix factorisations as Tate fixed-points \cite{p}. The fixed-point category $C[\![t]\!]$ 
of sections of $C$ over the classifying space $B^2\bZ =
\bC\bP^\infty$ can be built from the cellular approximations $\bC\bP^n$, with cohomology $C[\![t]\!]/t^n$. 
In that setting, $\deg t=2$, so relevant there are the variants $\fra^{-1}, \fra^{-1}_{<n}$ of 
our Lie algebras in degree $(-1)$, which are the (complexified) Lie homotopy models of those spaces. 
Concretely, a $\fra^{-1}$-action is the Lie algebra trivialisation of the trivial action of the circle group, 
descending it to a interesting $B\bZ$-action. (A one-parameter family of actions arises by scaling $t$.)  
Deformation theory as expounded here would correspond to the topological action of a $K(\bC,-1)$, 
a circle dimensionally shifted by $2$. The distinction is concealed in the 
$2$-periodic case.   

The fixed-point perspective clarifies why $C_0\sim C$ is an overoptimistic expectation. Recall that  
a compact group acting on a space $X$ will act topologically on the cohomology $H^*(X)$; the homotopy 
invariants in the latter form the equivariant cohomology $H^*_G(X)$. However, we can recover $H^*(X)\simeq H^*_G(X) 
\otimes^L_{H^*(BG)} \bZ$ only for nilpotent actions; this may fail on $\pi_0G$, and indeed the 
statement fails in general for finite groups. For actions of connected groups on categories, $\pi_1G$ can 
obstruct the statement as well.  
\end{remark}

\end{document}